\documentclass{amsart}
\usepackage{tikz-cd}
\usepackage{amsmath, amssymb}
\usepackage{array}
\usepackage{float}
\usepackage{mathtools}

\newtheorem{theorem}{Theorem}[section]
\newtheorem{corollary}[theorem]{Corollary}

\newtheorem{proposition}[theorem]{Proposition}

\theoremstyle{definition}
\newtheorem{definition}[theorem]{Definition}

\theoremstyle{remark}
\newtheorem{remark}[theorem]{Remark}

\usepackage{thmtools}
\usepackage{thm-restate}
\usepackage{hyperref}
\usepackage{cleveref}

\newcommand{\PP}{\mathbb{P}}



\begin{document}
\pagestyle{plain}

\title{Log Algebraic Hyperbolicity of $\overline{M}_{0,n}$}

\author[J. Wang]{Jiahe Wang}
\address{UCLA Mathematics Department, Box 951555, Los Angeles, CA 90095-1555, USA
}
\email{jiahewang@math.ucla.edu}

\begin{abstract}
We show that the moduli space of stable n-pointed rational curves $\overline{M}_{0,n}$ with its boundary $\Delta$ is algebraically hyperbolic.
\end{abstract}

\maketitle

\section{Introduction}

\begin{definition}
Given a complex projective variety $X$ and an effective divisor $D_X$, we say that the pair $(X,D_X)$ is \emph{algebraically hyperbolic} if there exists a uniform bound $\epsilon>0$ and an ample divisor $A$ such that for any integral curve $C$ in $X$ not contained in $D_X$, we have that \[2g(C)-2+i(C,D_X) \geq \epsilon\cdot \mathrm{deg}_A C\]where $g(C)$ is the geometric genus of $C$ and $i(C,D_X) = |f^{-1}(D_X)|$ where $f:\widehat{C}\rightarrow X$ the normalization map. We say $X$ is \emph{algebraically hyperbolic} if $X$ with the empty set is algebraically hyperbolic. 
\end{definition}
Given two pairs $(X,D_X)$, $(Y,D_Y)$ related by some map, one may ask whether the algebraic hyperbolicity of $(Y,D_Y)$ implies the algebraic hyperbolicity $(X,D_X)$. To establish this for $(X,D_X)$, one may consider if $X$ introduces more curves compared to $Y$, how genus and degree change under the map to $X$, and how the intersection behavior of $(X,D_X)$ compares with that of $(Y,D_Y)$. For example, for a finite étale cover $f:X\rightarrow Y$, one may show that the algebraic hyperbolicity of the base $Y$ implies the algebraic hyperbolicity of $X$. 

In the case of a smooth blow-up, the situation is straightforward as the genus and intersection numbers in many cases remain unchanged, so only the degree needs to be considered. We will see that this argument applies to $\overline{M}_{0,n}$, the moduli space of stable $n$-pointed rational curves, and its boundary, showing that the pair is algebraically hyperbolic.

We begin by recalling some basics of $\overline{M}_{0,n}$, its construction as a sequence of blow-ups from $(\PP^1)^{n-3}$ and $\mathbb P^{n-3}$, its boundary divisor $\Delta$, and its ample divisor related to the boundary. We then examine how log algebraic hyperbolicity behaves under blow-ups. Finally, we examine the algebraic hyperbolicity of the base pair for $(\PP^1)^{n-3}$ and $\mathbb P^{n-3}$. Together they yield the following:
\begin{theorem}
$(\overline{M}_{0,n}, \Delta)$ is algebraically hyperbolic. For $n\geq 4$, we may take $\epsilon = \frac 1{(n-3)^2}$ with respect to the polarization $K_{\overline{M}_{0,n}}+\Delta$.
\end{theorem}

Any curve $C$ inside $\overline{M}_{0,n}$ not entirely contained in the boundary $\Delta$ gives a one-parameter family of stable rational curves with $n$ markings whose general element is smooth. The algebraically hyperbolic condition of $(\overline{M}_{0,n}, \Delta)$ means that $2g(C)-2+i(C,\Delta)\geq\frac1{(n-3)^2}\cdot \mathrm{deg}_{K+\Delta}C$, where $i(C,\Delta)$ = $|f^{-1}(\Delta)|$ with $f:\widehat{C}\to X$ being the normalization map. So this gives a lower bound for the number of intersection points with the boundary. 

\subsection*{Acknowledgments}
I thank Wern Yeong for bringing up this project and for valuable teaching. I thank Dawei Chen for proposing this project and helpful comments. I am grateful to Burt Totaro for helpful comments and suggestions. I am grateful to Fernando Figueroa, Eric Riedl, and Lingyao Xie for helpful comments.

\section{Preliminaries}
\subsection{$\overline{M}_{0,n}$: the moduli space of stable \(n\)-pointed rational curves}\mbox{}\\

We recall basic facts about $\overline{M}_{0,n}$ and its boundary divisor $\Delta$. Standard references include \cite{Knudsen1983, keel1992intersection}.
\begin{definition}
For \(n\ge 3\),
\[
M_{0,n}
= \{(\mathbb P^1,p_1,\ldots,p_n) \mid p_i \neq p_j \text{ for } i\neq j\}/\mathrm{PGL}_2.
\]
Since an automorphism of \(\mathbb{P}^1\) moves three marked points to \(0,1,\infty\), we have \(\dim M_{0,n}=n-3\).
The Deligne–Mumford–Knudsen compactification, denoted \(\overline{M}_{0,n}\), is a smooth projective variety parametrizing stable \(n\)-pointed rational curves.
\end{definition}

\begin{definition}
The \emph{boundary divisor} is \(\Delta := \overline{M}_{0,n} \setminus M_{0,n}\), which is a simple normal crossings divisor in $\overline{M}_{0,n}$.
It decomposes as \[\Delta = \sum_{\substack{I \sqcup J = \{1, \ldots, n\} \\ 2 \le |I|, |J| \le n-2 \, / \sim}} D_{I,J},
\]
where $\sim$ identifies two partitions as the same if we swap $I$ and $J$. Each irreducible component \(D_{I,J}\) parametrizes stable curves with two rational components meeting at one node, with the marked points indexed by \(I\) on one component and those indexed by \(J\) on the other.
We let \(\overline{M}_{0,n}=\varnothing\) for \(n\le 2\). And we note that \(\overline{M}_{0,3}=\mathrm{pt}\) and \(\overline{M}_{0,4}\simeq \mathbb{P}^1\) with \(\Delta_{\overline{M}_{0,4}}=\{0,1,\infty\}\).
\end{definition}
We have the following theorem about the ampleness of the canonical divisor plus the boundary divisor:
\begin{theorem}[\cite{keel2005equationsbarm0n}, Theorem 1.1]
$K_{\overline{M}_{0,n}}+\Delta$ is very ample.
\end{theorem}

\subsection{Blow-up structures of $\overline{M}_{0,n}$}\mbox{}\\
$\overline{M}_{0,n}$ can be realized as a sequence of smooth blow-ups. This is realized in Kapranov \cite{kapranov1992chowquotientsgrassmanniani}, where $\overline{M}_{0,n}$ is constructed as blow-ups of $\mathbb{P}^{n-3}$. Keel \cite{keel1992intersection} realizes $\overline{M}_{0,n}$ as a sequence of blow-ups of $(\mathbb{P}^1)^{n-3}$. These models are particularly valuable for explicit computations of intersection theory and for understanding the birational geometry of $\overline{M}_{0,n}$. Both blow-up structures are stated in Hassett \cite{hassett2002modulispacesweightedpointed}.

\vspace{0.5cm}

We first consider the construction of Keel, which realizes $\overline{M}_{0,n}, n\geq 5$ as a sequence of blow-ups of $(\mathbb P^1)^{n-3}$ along loci of diagonals: 

For any $n\geq 5$, consider $(\mathbb P^1, p_1,...,p_n)$, where $p_1,...,p_n$ are distinct points. There exists a unique automorphism $\phi$ of $\mathbb P^1$ such that it maps 
$$ \phi: (p_1,p_2,p_3) \mapsto (0,1,\infty).$$
The image of the $(\mathbb P^1, p_1,...,p_n)$ in $M_{0,n}$ is determined by the points
\[
(\phi(p_4),\ldots,\phi(p_n)),
\]
and we obtain an embedding
\[
M_{0,n} \hookrightarrow (\mathbb{P}^1)^{n-3}.
\]

Let $\Delta_d$ denote the union of the dimension $d$ diagonals, i.e., the locus where at least 
$n-2-d$ of the points coincide. We will use this notation for both the locus in $(\mathbb{P}^1)^{n-3}$ and its proper transforms. Let
\[
F_0 = \mathrm{pr}_1^{-1}(0) \cup \mathrm{pr}_2^{-1}(0) \cup \cdots \cup \mathrm{pr}_{n-3}^{-1}(0)
\]
be the locus of points mapping to $0$ under one of the projections $\mathrm{pr}_j$;
we define $F_1$ and $F_\infty$ analogously.
Again, we use the same notation for proper transforms. We note that $M_{0,n}\cong (\mathbb P^1)^{n-3}\setminus (F_0\cup F_1\cup F_\infty\cup\Delta_{n-4})$. In the blow-up structure below, the boundary divisor $\Delta$ of $\overline{M}_{0,n}$ corresponds to the exceptional divisors plus the strict transform of $F_0\cup F_1\cup F_\infty\cup\Delta_{n-4}$.

\begin{theorem}[\cite{hassett2002modulispacesweightedpointed}, section 6.3]\label{thmk}
We factor $\rho : \overline{M}_{0,n} \to (\mathbb{P}^1)^{n-3}$ as a product of blow-ups.
Write $Y_0[n] = (\mathbb{P}^1)^{n-3}$ and define the first sequence of blow-ups as follows:

\begin{align*}
1 &: Y_1[n] \text{ is the blow-up along the intersection }\Delta_1 \cap (F_0 \cup F_1 \cup F_\infty);\\
2 &: Y_2[n]\text{ is the blow-up along the intersection }\Delta_2 \cap (F_0 \cup F_1 \cup F_\infty);\\
&\hspace{1cm}\vdots\\
n-4&: Y_{n-4}[n]\text{ is the blow-up along the intersection }\Delta_{n-4} \cap (F_0 \cup F_1 \cup F_\infty).
\end{align*}

The second sequence of blow-ups is:
\begin{align*}
n-3 &:Y_{n-3}[n] \text{ is the blow-up along } \Delta_1;\\
n-2&:Y_{n-2}[n]  \text{ is the blow-up along } \Delta_2;\\
&\hspace{1cm}\vdots\\
2n-9&:Y_{2n-9}[n]    \text{ is the blow-up along } \Delta_{n-5}.
\end{align*}

\end{theorem}

The second construction is given by Kapranov, which realizes $\overline{M}_{0,n}, n\geq5$ as a sequence of blow-ups of $\mathbb{P}^{n-3}$ at $n-1$ general points and the linear subspaces that they span. Moreover, the boundary divisor corresponds to the exceptional divisors plus the strict transforms of the $\binom{n-1}{n-3}$ hyperplanes in $\mathbb{P}^{n-3}$.

\begin{theorem}[\cite{hassett2002modulispacesweightedpointed}, section 6.2]\label{thm2}
Given $p_1,...,p_{n-1}$ points in general position in $\mathbb P^{n-3}$, $\overline{M}_{0,n}$ is a sequence of blow-ups of $\mathbb{P}^{n-3}$:
\begin{align*}
1 &: \text{ Blow up the points } p_1,\ldots,p_{n-1}\text{ to obtain }X_1[n];\\
2 &: \text{ Blow up lines spanned by pairs of the points }p_1,\ldots,p_{n-1} \text{ to obtain }X_2[n];\\
3 &: \text{ Blow up 2-planes spanned by triples of the points to obtain }X_3[n];\\
 & \hspace{0.5cm}\vdots\\
n-4 &: \text{ Blow up }(n-5)\text{-planes spanned by }(n-4)\text{-tuples of the points to obtain }X_{n-4}[n].
\end{align*}

We interpret the exceptional divisors of the induced maps. For each partition $\{1,\ldots,n\} = I \cup J$ with $I = \{i_1 = n, i_2,\ldots,i_r\}$, $J = \{j_1,\ldots,j_{n-r}\}$ and $2 \leq r \leq n-2$, consider the corresponding boundary divisor in $\overline{M}_{0,n}$: \[ D_{I,J} \simeq \overline{M}_{0,r+1} \times \overline{M}_{0,n-r+1}. \]

The divisors $D_{I,J}$ with $|I| = r < n-2$ are the exceptional divisors of $X_{r-1}[n] \to X_{r-2}[n]$. The divisors $D_{I,J}$ with $|I| = n-2$ are proper transforms of the hyperplanes of $X_0[n] \simeq \mathbb{P}^{n-3}$ spanned by $p_{i_2},\ldots,p_{i_{n-2}}$.
\end{theorem}

\subsection{Log algebraic hyperbolicity}

\begin{definition}[{\cite[Definition 1.10]{chen2001algebraichyperbolicitylogvarieties}}]
     Given a complex projective variety $X$ and an effective divisor $D_X$, we say that the pair $(X,D_X)$ is \emph{algebraically hyperbolic} if there exists a uniform bound $\epsilon>0$ and an ample divisor $A$ such that for any integral curve $C$ in $X$ not contained in $D_X$, we have that \[2g(C)-2+i(C,D_X) \geq \epsilon\cdot \mathrm{deg}_A C\]where $g(C)$ is the geometric genus of $C$ and $i(C,D_X) = |f^{-1}(D_X)|$ where $f:\widehat{C}\rightarrow X$ the normalization map. We say $X$ is \emph{algebraically hyperbolic} if $X$ with the empty set is algebraically hyperbolic. 
\end{definition}

Algebraic hyperbolicity was introduced as an algebraic analogue to Kobayashi hyperbolicity for complex manifolds \cite{Demailly1997}. A complex manifold \(X\) is said to be Kobayashi hyperbolic if its Kobayashi pseudometric is nondegenerate, and Brody hyperbolic if every entire map \(f:\mathbb{C}\to X\) is constant. For smooth projective varieties, Kobayashi hyperbolicity implies algebraic hyperbolicity and the converse is conjectured to be true \cite{Demailly1997}. For the log case, Kobayashi hyperbolicity of the complement with the extra condition of being hyperbolically imbedded implies algebraic hyperbolicity \cite[Theorem 5]{pacienza2006logarithmickobayashiconjecture}. Brody hyperbolicity is in general weaker than Kobayashi hyperbolicity but is equivalent to Kobayashi hyperbolicity for compact manifolds \cite{Brody1978}.

Algebraic hyperbolicity or log algebraic hyperbolicity is important because it gives information about the genus and degree of curves inside a variety, as well as intersection behavior for the log case. In particular, if $X$ is algebraically hyperbolic, then we know that $X$ does not contain any rational or elliptic curves. If $(X, D_X)$ is algebraically hyperbolic, then the complement \(X \setminus D_X\) contains no images of \(\mathbb{P}^1\) or \(\mathbb{A}^1\), and no elliptic curves. Moreoever, for the log case, the inequality gives a lower bound of the intersection number $i(C,D_X)$, which may be geometrically meaningful.

Log algebraic hyperbolicity was introduced by Chen \cite{chen2001algebraichyperbolicitylogvarieties} as a way to analyze the Kobayashi hyperbolicity of the complement, towards a result of the Kobayashi Conjecture. Log algebraic hyperbolicity is also important because it appears in the algebraic version of the Green-Griffiths-Lang Conjecture \cite{ascher2024algebraicgreengriffithslangconjecturecomplements}.

We mention some results on hyperbolicity of moduli spaces. In \cite{viehweg2003brodyhyperbolicitymodulispaces}, the authors examine Brody hyperbolicity of canonically polarized manifolds. In \cite{popa2016viehwegshyperbolicityconjecturefamilies}, the authors show that the base spaces are of log general type for families with maximal variation and fibers of general type.

\section{Log Algebraic Hyperbolicity of $\overline{M}_{0,n}$}

In this section we examine two cases  where algebraic hyperbolicity is passed on under sequences of blow-ups. One is remarked by Chen \cite{chen2001algebraichyperbolicitylogsurfaces} and is completely general for any birational pairs; the other one is more specific to $\overline{M}_{0,n}$. In the second case we have a better description of the polarization, namely the ample line bundle defining degree.

In \cite{chen2001algebraichyperbolicitylogsurfaces}, Chen defines log algebraic hyperbolicity and briefly remarks that log algebraic hyperbolicity does not depend on the pair, instead it only depends on the complement. We provide a short proof.

\begin{proposition}[{\cite[p.~4]{chen2001algebraichyperbolicitylogsurfaces}}]
If $(X, D_X)$ and $(Y, D_Y)$ satisfy $X \setminus D_X \cong Y \setminus D_Y$, then $(X, D_X)$ is algebraically hyperbolic if and only if $(Y, D_Y)$ is.
\end{proposition}

\begin{proof}
By resolving to a common log resolution of $(X,D_X)$ and $(Y,D_Y)$, $f:S\to X$ and $g:S\to Y$ such that $ Supp(f_*^{-1} D_X + E_f) =Supp(g_*^{-1} D_Y + E_g)$ (consider the log resolution of the closure of the graph $\Gamma \subset X\times Y$ of the isomorphism $X\backslash D_X\cong Y\backslash D_Y$), it suffices to consider the case where we obtain $(S, D_S)$ from $(X,D_X)$ through a sequence of blow-ups, where $D_S = f_*^{-1} D_X + E_f$, and we show that $(S,D_S)$ is algebraically hyperbolic if and only if $(X, D_X)$ is algebraically hyperbolic.

We note that log algebraic hyperbolicity is independent of the ample divisor. We consider an ample line bundle $L$ on $X$, then the pullback $A = f^*L$ is base-point-free and big but not necessarily ample. There is a positive linear combination  $E$ of the exceptional divisors such that $-E$ is $f$-ample \cite[section 44]{Kollar-CanonicalModels-40}.
By \cite[Proposition 1.7.10]{LazarsfeldI}, there exists $M\geq 0$, such that $mf^*L-E$ is ample for all $m\geq M$. Then choose any $a\in (0,1/M]$, we have that $L_S := f^*L - a E$ is ample.

For any curve $C$ not entirely contained in $D_S$, we have that
\[
(1-aM)(A \cdot C) \le L_S \cdot C \le A \cdot C.
\]

If $(X,D_X)$ is algebraically hyperbolic with some $\epsilon$ with respect to $L$, then $(S,D_S)$ is algebraically hyperbolic with respect to the same $\epsilon$ and $L_S = f^*L - aE$. If $(S, D_S)$ is algebraically hyperbolic with respect to $L_S$ for some $\epsilon$, then $(X, D_X)$ is algebraically hyperbolic with respect to $L$, for $\epsilon(1-aM)$ which can be made arbitrarily closed to $\epsilon$.
\end{proof}
\begin{remark}
    In the previous proposition, even though we have shown that algebraic hyperbolicity is passed on between pairs with isomorphic complements, but it doesn't yield a formula of the polarization. Finding an $\epsilon$ under some known polarization is important as it gives geometric meaning to the moduli space. For example, in our case, for any curve $C$ inside $\overline{M}_{0,n}$ not entirely contained in the boundary $\Delta$, $C$ gives a one-parameter family of stable rational curves with $n$ markings whose general element is smooth. The algebraically hyperbolic condition of $(\overline{M}_{0,n}, \Delta)$ means that $2g(C)-2+i(C,\Delta)\geq \epsilon\cdot \mathrm{deg}_{K+\Delta}C$, where $i(C,\Delta)$ = $|f^{-1}(\Delta)|$ with $f:\widehat{C}\to \overline{M}_{0,n} $ being the normalization map. So this gives a lower bound for the number of intersection points with the boundary. 

\end{remark}
We work out the second case where we gain more information on the polarization. Consider a sequence of smooth blow-ups $X=X_0\to X_1\to X_2\to...\to X_n=Y$, $f_i:X_i\to X_{i+1}$, $f=f_{n-1}\circ f_{n-2}\circ...\circ f_0$. Given a divisor $D_Y$ on $Y$, for each $i$ we let $D_i = f_{i*}^{-1}D_{i+1}+E_i$, then $D_X = f_*^{-1}D_Y+\sum_i E_i$.
\begin{proposition}\label{prop1}
Consider a sequence of smooth blow-ups as above,
$f : (X, D_X) \to (Y, D_Y)$, 
where $D_X = f_*^{-1} D_Y + \sum_i E_i$. If each blow-up center $Z_i\subset X_i$, $\mathrm{mult}_{Z_i} D_{i} \geq \mathrm{codim}(Z_i)$, and $K_X+D_X$ and $K_Y+D_Y$ are ample, then $(Y, D_Y)$ being algebraically hyperbolic with respect to the polarization $K_Y+D_Y$ implies that
$(X, D_X)$ is algebraically hyperbolic with the same $\epsilon$ under the polarization $K_X+D_X$.
\end{proposition}

\begin{proof}
It suffices to consider a single smooth blow-up $f:(X,D_X)\to (Y,D_Y)$ at the center $Z\subset Y$. For any curve $\mathcal C$ in $X$ not contained in $D_X$, the pushforward $C := f_\ast \mathcal C$ is a curve in $Y$ not contained in $D_Y$, and $\mathcal C = \widetilde{C}$, the strict transform of $C$. Since $(Y, D_Y)$ is algebraically hyperbolic, there exists $\epsilon > 0$ such that for any curve $C$ not contained in $D_Y$:
\[ 2g(C) - 2 + i(C,D_Y) \geq \epsilon (K_Y + D_Y) \cdot C , \]
where $g(C)$ is the geometric genus and $i(C,D_Y)$ counts the number of points in the normalization $\widehat{C}$ mapping to $D_Y$.

Under the blow-up, we have $g(\widetilde{C}) = g(C)$, and $i(\widetilde{C},D_X) \geq i(C,D_Y)$ since $D_X = f_*^{-1} D_Y + E$. We analyze how the degree changes with respect to the ample divisor $K_{X} + D_X$:
\[(K_{X} + D_X) \cdot \widetilde{C}  = ((f^*K_{Y} + (\mathrm{codim}(Z)-1)E) + (f^*D_Y  - \mathrm{mult}_Z(D_Y)E+E))\cdot \widetilde{C} , \]
\[= (f^*(K_{Y} + D_Y) - (\mathrm{mult}_Z(D_Y)-\mathrm{codim}(Z)) E)\cdot \widetilde{C}  \leq f^*(K_{Y} + D_Y)\cdot \widetilde{C}  = (K_{Y} + D_Y) \cdot C, \]
Therefore:
\[ 2g(\widetilde{C}) - 2 + i(\widetilde{C},D_X) \geq \epsilon (K_{X} +D_X) \cdot  \widetilde{C}. \]

\end{proof}

Recall that, for $n\geq5$, $\overline{M}_{0,n}$ can be realized as a sequence of blow-ups of $(\mathbb P^1)^{n-3}$ along loci of diagonals. We have the map  $f:(\overline{M}_{0,n},\Delta)\to ((\mathbb P^1)^{n-3},D)$, where $D=F_0\cup F_1\cup F_\infty\cup\Delta_{n-4}$. We have the polarization  $K_{\overline{M}_{0,n}}+\Delta$ for $\overline{M}_{0,n}$ and the polarization $K_{(\mathbb P^1)^{n-3}}+D$ for $(\mathbb P^1)^{n-3}$.

\begin{theorem}
$(\overline{M}_{0,n}, \Delta)$ is algebraically hyperbolic. For $n\geq 4$, we may take $\epsilon = \frac 1{(n-3)^2}$ with respect to the polarization $K_{\overline{M}_{0,n}}+\Delta$.
\end{theorem}

\begin{proof}
For $n\leq 4$, $(\overline{M}_{0,n},\Delta)$ is algebraically hyperbolic. For $n\geq5$, we have the map  $f:(\overline{M}_{0,n},\Delta)\to((\mathbb P^1)^{n-3}, D)$, where $D =F_0\cup F_1\cup F_\infty\cup\Delta_{n-4}$. By Proposition \ref{propk} in the next section, we know that $((\mathbb P^1)^{n-3}, D)$ is algebraically hyperbolic with $\epsilon  = \frac 1{(n-3)^2}$ with respect to the polarization $K_{(\PP^1)^{n-3}}+D$. Also the morphism $f$ satisfies the condition in Proposition \ref{prop1} as follows:

We verify that every blow-up center appearing in Theorem~\ref{thmk}
satisfies the inequality 
\(\mathrm{mult}_{Z_i}D_i \ge \mathrm{codim}(Z_i)\).
For any subset \(I\subset\{4,\ldots,n\}\) and 
\(\star\in\{0,1,\infty\}\),
we denote
\[
\Delta_{I,\star} := \{\,x_i=\star \text{ for all } i\in I\,\},
\] 
\[\Delta_I := \{\,x_i=x_j\text{ for all }i,j\in I\,\}.\]
Then we have that, \[\Delta_d \cap (F_0\cup F_1\cup F_\infty ) = \bigcup_{|I|=n-2-d, \star \in \{0,1,\infty\}} \Delta_{I,\star},\] \[\Delta_d = \bigcup_{|I|=n-2-d} \Delta_I.\]

We analyze the multiplicity of the boundary divisor 
\(D = F_0\cup F_1\cup F_\infty\cup \Delta_{n-4}\) along these loci.
For the first type of center,
\(\Delta_{I,\star}\subset \Delta_d \cap(F_0\cup F_1\cup F_\infty)\), 
the locus is contained in $\{x_i = \star\}\subset F_\star, i\in I$ and the pairwise diagonals 
\(\Delta_{ij}\subset \Delta_{n-4}\) for \(i,j\in I\); hence, in total, 
\[
\mathrm{mult}_{\Delta_{I,\star}}(D)
= |I| + \binom{|I|}{2}
\quad\text{and}\quad
\mathrm{codim}(\Delta_{I,\star})=|I|.
\]
Therefore 
\(\mathrm{mult}_{\Delta_{I,\star}}(D)\ge \mathrm{codim}(\Delta_{I,\star})\) for \(|I|\ge2\).

For the second type of centers, the pure diagonals \(\Delta_I\),
we have 
\(\mathrm{codim}(\Delta_I)=|I|-1\). The locus is contained in the pairwise diagonals
\(\Delta_{ij}\) for \(i,j\in I\),
so for $|I|\geq 3$,
\[
\mathrm{mult}_{\Delta_I}(D)=\binom{|I|}{2}\ge |I|-1=\mathrm{codim}(\Delta_I).
\]
Moreover, in the blow-up order given in Theorem~\ref{thmk},
none of the exceptional divisors produced in earlier steps
contains the later centers. 
Hence the multiplicity along each center is computed only from the 
strict transforms of the original boundary components and 
remains unchanged during the process.
Consequently, every blow-up center in the Keel sequence satisfies 
\(\mathrm{mult}_{Z_i} D_i \ge \mathrm{codim}(Z_i)\),
as required in Proposition~\ref{prop1}.
\end{proof}

\begin{remark}
   One may consider the blow-up structure from $\mathbb P^{n-3}$ and derive the log algebraic hyperbolicity. We have the blow-up map $f:(\overline{M}_{0,n}, \Delta)\to (\mathbb P^{n-3}, H)$, where $H$ is the union of $\binom{n-1}{n-3}$ hyperplanes: 
   
  By Proposition \ref{prop3} in the next section, we know that $(\mathbb P^{n-3}, H)$ is algebraically hyperbolic, and also the morphism $f$ satisfies the condition in Proposition \ref{prop1} as: 
For each blow-up center $Z$, i.e., a linear space of codimension $r\geq 2$ given by $n-2-r$ of the $n-1$ general points, from the blow-up construction in Theorem \ref{thm2}, the multiplicity along the linear space is equal to the multiplicity of $H$ along the linear space in $\mathbb P^{n-3}$, which is given by the number of hyperplanes containing the linear space:
\[ \mathrm{mult}_Z H = \binom{n-1 - (n-2-r)}{n-3 - (n-2-r)} = \binom{r+1}{r-1} = \frac{r(r+1)}{2} > r \text{ for } r \geq 2. \]
\end{remark}
\vspace{0.3cm}

We also state the following relevant case, which uses the anti-canonical bundle as the polarization, though it only shows log algebraic hyperbolicity for $\overline{M}_{0,5}$ and $\overline{M}_{0,6}$.

\begin{proposition}
Let $f: (X,D_X) \to (Y,D_Y)$ be a sequence of smooth blow-ups, where $D_X=f_*^{-1}D_Y+\sum_iE_i$. If $-K_Y$ is ample and $X$ is log Fano, i.e., $-(K_X+D)$ is ample for some effective divisor $D$, then $(Y,D_Y)$ being algebraically hyperbolic implies $(X,D_X\cup D)$ is algebraically hyperbolic.
\end{proposition}

\begin{proof}
Since $(Y,D_Y)$ is algebraically hyperbolic, there exists $\epsilon > 0$ such that:
\[ 2g(C) - 2 + i(C,D_Y) \geq \epsilon (-K_Y) \cdot C . \]
We have $-K_X = -f^*K_Y - \sum_ia_iE_i$ where $a_i>0$. Thus:
\[ -K_X \cdot \widetilde{C} = -f^*K_Y \cdot \widetilde{C} - \sum_i a_iE_i \cdot \widetilde{C}  \leq -K_Y \cdot C.\]

As $X$ is log Fano, there exists an effective divisor $D$ such that $-(K_X + D)$ is ample. Therefore:
\[ 2g(\widetilde{{C}}) - 2 + i(\widetilde{C},D_X\cup D) \geq \epsilon (-K_X) \cdot \widetilde{C} \geq \epsilon (-(K_X + D)) \cdot \widetilde{C} , \]
so $(X,D_X\cup D)$ is algebraically hyperbolic. 

Note that by $D_X\cup D$, we mean an effective divisor supported on $D_X\cup D$. In the definition of log algebraic hyperbolicity only the support of the divisor is relevant.
\end{proof}

\begin{corollary}
$\overline{M}_{0,5}$ and $\overline{M}_{0,6}$ with their boundary divisors are algebraically hyperbolic.
\end{corollary}

\begin{proof}
The space $\overline{M}_{0,n}$ is log Fano if and only if $n \leq 6$. The space $\overline{M}_{0,5}$ is Fano, so it with its boundary is algebraically hyperbolic. For sufficiently small $\varepsilon>0$, the divisor $\bigl(K_{\overline{M}_{0,6}} + \varepsilon \Delta)$ is ample \cite{keel1996contractibleextremalraysoverlinem0n}. Therefore, by the previous proposition, we have that $(\overline{M}_{0,6},\Delta)$ is algebraically hyperbolic. 
\end{proof}

\section{Algebraic hyperbolicity of $((\mathbb P^1)^{n-3},D)$ and $(\mathbb P^{n-3},H)$}
\subsection{Algebraic hyperbolicity of $((\mathbb P^1)^{n-3},D)$}\mbox{}\\
Recall that $M_{0,n} \hookrightarrow (\mathbb{P}^1)^{n-3}$ and $D = (\mathbb{P}^1)^{n-3}\setminus M_{0,n}  = F_0\cup F_1\cup F_\infty\cup \Delta_{n-4}$.

\begin{proposition}\label{propk}
The pair $((\mathbb P^1)^{n-3},D)$ is algebraically hyperbolic and we have $\epsilon = \frac 1{(n-3)^2}$ under the polarization $K_{(\mathbb P^1)^{n-3}}+D$.
\end{proposition}

\begin{proof}
Let $X = (\mathbb P^1)^{n-3}$, $D_0 = F_0\cup F_1\cup F_{\infty}$, and $m=n-3$. Let $f:\widehat C\to (\PP^1)^{m}$ be the normalization map of some $C$ not entirely contained in $D_0$, and write $f_j:=\mathrm{pr}_j\circ f:\widehat C\to\PP^1$.
For each $j$, set $S_j:=(f_j)^{-1}\{0,1,\infty\}_{\mathrm{red}}$. Apply log Riemann-Hurwitz to 
\( f_j : (\widehat{C}, S_j) \to (\PP^1, \{0,1,\infty\}) \):
\[
K_{\widehat{C}} + S_j \ge f_j^*(K_{\PP^1} + \{0,1,\infty\}),
\]
so taking degrees,
$$
2g(\widehat{C}) - 2 + |S_j| \ge \deg f_j^*(K_{\PP^1} + \{0,1,\infty\}).
$$

Sum over \( j = 1, \ldots, m \):
$$
m(2g - 2) + \sum_{j=1}^m |S_j|
\ge \sum_{j=1}^m \deg f_j^*(K_{\PP^1} + \{0,1,\infty\})
= \deg_{K_X + D_0}(C).
$$

To relate \( \sum |S_j| \) to \( |S| = i(C,D_0) \), we note that each point of \( S \) lies in at most \( m \) of the \( S_j \)'s, hence
\[
\sum_{j=1}^m |S_j| \le m |S| = m\, i(C,D_0).
\]
So we have that
$$
2g - 2 + i(C,D_0)\ge  \frac 1m\deg_{K_X + D_0}(C).
$$ 
Let $H_i=\mathrm{pr}_i^*\mathcal O_{\PP^1}(1)$.
We have
$
K_X \;\sim\; -2\sum_{i=1}^m H_i
$
and 
$
D_0 = F_0+F_1+F_\infty \;\sim\; 3\sum_{i=1}^m H_i.
$
Each big diagonal $\Delta_{ij}=\{x_i=x_j\}$ is a divisor with class
$
\Delta_{ij}\;\sim\; H_i+H_j.
$
Summing over all $1\le i<j\le m$ gives
$
\sum_{1\le i<j\le m}\Delta_{ij} \;\sim\; (m-1)\sum_{i=1}^m H_i.
$
Therefore
$
K_X+D\;\sim\; m\sum_{i=1}^m H_i
\;=\; m (K_X+D_0).
$
Hence for any integral curve $C\subset X$ not entirely contained in $D$,
\[
\deg_{K_X+D}(C) \;=\; m\,\deg_{K_X+D_0}(C).
\]
So we have,
\[
2g-2+i(C,D)\ \ge\ 2g-2+i(C,D_0)\ \ge\ \frac 1m\deg_{K_X+D_0}(C)
\;=\; \frac 1{m^2}\,\deg_{K_X+D}(C).
\]
This proves algebraic hyperbolicity of $\bigl((\PP^1)^m,D\bigr)$ with
$\epsilon=\frac1{m^2}$ for the polarization $K_X+D$.
\end{proof}

\subsection{Algebraic hyperbolicity of $(\mathbb P^{n-3},H)$}\mbox{}\\
On a different note, one may show that $(\mathbb{P}^{n-3}, H)$ is algebraically hyperbolic using Kobayashi hyperbolicity for the complement $M_{0,n}$ of the pair derived from some hyperplane-arrangement-type arguments of $\mathbb P^{n-3}$. 

\begin{definition}[\cite{Ko98}, p.~138]
A set of hyperplanes $H_1,\ldots,H_N$ in $\mathbb{P}^n_{\mathbb{C}}$ is in \emph{hyperbolic configuration} (condition (h)) if every projective line $\ell$ in $\mathbb{P}^n_{\mathbb{C}}$ intersects $\bigcup H_i$ in at least three points. It is in \emph{hyperbolic-imbedding configuration} (condition (hi)) if every projective line $\ell$ intersects $\bigcup_{H_i \not\supset \ell} H_i$ in at least three points.
\end{definition}

The following theorems connect (hi) to Kobayashi hyperbolicity and log algebraic hyperbolicity:

\begin{theorem}[\cite{Ko98}, p.~143]\label{thm}
Given hyperplanes $H_1,\ldots,H_N$ in $\mathbb{P}^n_{\mathbb{C}}$, set $X = \mathbb{P}^n_{\mathbb{C}} \setminus \bigcup H_i$. If the hyperplanes are in hyperbolic-imbedding configuration, then $X$ is complete hyperbolic and hyperbolically imbedded in $\mathbb{P}^n_{\mathbb{C}}$. Conversely, if $X$ is hyperbolically imbedded in $\mathbb{P}^n_{\mathbb{C}}$, then the hyperplanes are in hyperbolic-imbedding configuration.
\end{theorem}

\begin{theorem}[\cite{pacienza2006logarithmickobayashiconjecture}, Theorem 5]
Let $X$ be a projective manifold and $D$ an effective divisor such that $X \setminus D$ is hyperbolic and hyperbolically imbedded. Then there exists $\epsilon > 0$ such that for any integral curve $C \subset X$ with $C \not\subset D$:
\[ 2g(C) - 2 + i(C,D) \geq \epsilon\cdot \deg_A(C), \]
where $g(C)$ is the geometric genus of $C$, $i(C,D_X) = |f^{-1}(D_X)|$ where $f:\widehat{C}\rightarrow X$ the normalization map, $A$ an ample divisor.
\end{theorem}

It remains to show that $(\mathbb{P}^{n-3}, H)$ satisfies (hi). Recall that $H$ consists of $\binom{n-1}{n-3}$ hyperplanes spanned by sets of $n-3$ points from $n-1$ points in linear general position, meaning no point lies in the span of any subset of $n-3$ other points, or say any $n-2$ points form a basis for $\mathbb C^{n-2}$.

\begin{proposition}\label{prop3}
$(\mathbb{P}^{n-3}, H)$ satisfies (hi) for all $n \geq 5$, and therefore is algebraically hyperbolic.
\end{proposition}

\begin{proof}

Note that if $S$ and $S'$ are two subsets of the $n-1$ general points $\{p_1,p_2,...,p_{n-1}\}$ such that $S\cup S'$ contains at most $n-2$ points, then $\text{span}(S)\cap \mathrm{span}(S') = \mathrm{span}{(S\cap S')}$.

Suppose for contradiction that a line only intersects $\cup_{l \not\subseteq H_i} H_i$ at two points $a_1$ and $a_2$. We can always find a hyperplane $H$ such that $H$ avoids $a_2$, so it must contain $a_1$. This is because, if not then all $\mathrm{span}{S_i}$, where $S_i$ are $(n-3)$-subsets among the first $n-2$ points, would contain $a_2$. But $\cap \mathrm{span}(S_i) = \mathrm{span}(\cap S_i) = \mathrm{span}(\emptyset) = 0$. We can assume $H = \mathrm{span}\{p_1, p_2, \dots, p_{n-3}\}$.

We can remove a point among $\{p_1, \dots, p_{n-3}\}$ such that the span avoids $a_1$. This is because if not, then $a_1 \in \cap \mathrm{span}(S_i)$ where $S_i$ are $(n-4)$-subsets of $\{p_1, \dots, p_{n-3}\}$. We have that $\cap \mathrm{span}(S_i) = \mathrm{span}(\emptyset) = 0$. Say we can remove $p_1$, and consider the following two hyperplanes:

\[
H_1 = \mathrm{span}\{p_2, \dots, p_{n-3}\} \cup \{p_{n-2}\},
\]
\[
H_2 = \mathrm{span}\{p_2, \dots, p_{n-3}\} \cup \{p_{n-1}\}.
\]

Then $H_1$ and $H_2$ avoid $a_1$, because if not, then $H \cap H_i = \mathrm{span}\{p_2, \dots, p_{n-3}\}$ contains $a_1$ for $i=1\text{ or }2$, which is not true. So $H_1$ and $H_2$ must contain $a_2$. Thus, $a_2 \in H_1 \cap H_2 = \mathrm{span}\{p_2, \dots, p_{n-3}\} \subset H$, but $H$ avoids $a_2$.
\end{proof}

\begin{remark}
Showing the algebraic hyperbolicity of $(\PP ^{n-3}, H)$ using the method above does not explicitly yield an $\epsilon$. To find an $\epsilon$, one may attempt to classify all possible curves $f:\widehat C\to \mathbb{P}^{n-3}$ not entirely contained in $H$ and examine the relationship between $|f^{-1}(H)|$ and its degree.

For example, when $n=5$, we have that $H$ in the pair $(\mathbb{P}^2, H)$ is the union of six lines passing through four general points $p_1, p_2, p_3, p_4$. At first glance, it may seem hopeful as one can relate $|f^{-1}(H)|$ to the degree of $C$. In particular, for any curve inside $\mathbb{P}^2$ if the intersection of the curve with $H$ is transverse then we get that $|f^{-1}(H)| = 6d$. So the inequality becomes $2g(C) - 2 + 6d \geq \epsilon 3d$. We can take $\epsilon =1$, then $2g(C)-2+3d \geq -2 +3d\geq 0$. But when the intersection is not transverse, the situation is complex. Consider a general line intersecting with a curve at a node. Then in this case, we have $|f^{-1}(L)| = 2$ so it is the same as the number of intersections without the singularity. Or more explicitly, the intersection multiplicity at the nodal point is $2$, and the normalization process separates the nodal point into $2$, so the singularity does not reduce $|f^{-1}(L)|$. However, consider another case where we have a general line intersecting with a curve at a cusp; the intersection multiplicity is $2$, but under the normalization, the preimage of the cusp is just one point, so the singularity reduces $|f^{-1}(L)|$ by $1$. More generally, one depending on the singularity and the intersection behavior, the $|f^{-1}(L)|$ may even be smaller.

Under this strategy, the best one can do is to examine certain extreme cases, and argue that $\epsilon$ is at most this number. For example we may consider a conic passing through the points $p_1, p_2, p_3, p_4$. Then $|f^{-1}(H)| = 4$, $2g-2+|f^{-1}(H)| = -2+4=2 = \frac 13\deg_{K+H} C$. So $\epsilon$ is at most $\frac 13$.
\end{remark}

\bibliographystyle{plain}
\bibliography{reference}

\end{document}